\documentclass[11pt,reqno]{amsart}
\usepackage{graphicx} 
\usepackage{hyperref}
\usepackage{fourier}
\usepackage{fullpage}
\usepackage{amsmath,amssymb,amsthm}
\usepackage[shortlabels]{enumitem}
\usepackage[mathscr]{euscript}
\usepackage{dutchcal}
\usepackage{upgreek}
\usepackage{comment}
\usepackage{csquotes}
\usepackage{tikz}
\usepackage[all]{xy}
\allowdisplaybreaks

\makeatletter
\def\resetMathstrut@{%
  \setbox\z@\hbox{%
    \mathchardef\@tempa\mathcode`\(\relax
    \def\@tempb##1"##2##3{\the\textfont"##3\char"}%
    \expandafter\@tempb\meaning\@tempa \relax
  }%
  \ht\Mathstrutbox@1.2\ht\z@ \dp\Mathstrutbox@1.2\dp\z@
}
\makeatother

\newtheorem{theorem}{Theorem}
\newtheorem{lemma}[theorem]{Lemma}
\newtheorem{prop}[theorem]{Proposition}

\newtheorem{corollary}[theorem]{Corollary}

\newtheorem{definition}[theorem]{Definition}

\newtheorem{remark}[theorem]{Remark}

\newcommand{\R}{\mathbb{R}}
\newcommand{\Z}{\mathbb{Z}}
\newcommand{\N}{\mathbb{N}}

\DeclareMathOperator{\Lip}{Lip}

\renewcommand{\le}{\leqslant}
\renewcommand{\ge}{\geqslant}

\renewcommand{\setminus}{\smallsetminus}
\renewcommand{\subset}{\subseteq}

\newcommand{\ee}{\mathsf{e}}

\newcommand{\eqdef}{\stackrel{\mathrm{def}}{=}}
\newcommand{\cZ}{\mathcal{Z}}

\newcommand{\X}{\mathbf X}
\newcommand{\Id}{\mathsf{Id}}
\newcommand{\cj}{\mathcal{j}}

\newcommand{\NN}{\mathcal{N}}
\newcommand{\MM}{\mathcal{M}}

\newcommand{\TT}{\mathscr{T}}

\renewcommand{\SS}{\mathscr{S}}
\newcommand{\sub}{\mathscr{C}}

\newcommand{\bX}{\mathbf{X}}

\newcommand{\bY}{\mathbf{Y}}
\newcommand{\bZ}{\mathbf{Z}}

\newcommand{\f}{\varphi}
\renewcommand{\d}{\delta}
\newcommand{\ud}[0]{\,\mathrm{d}}

\renewcommand{\nu}{\upnu}

\renewcommand{\r}{\otherrho}

\title{De-H\"oldering factorization}

\date{}

\author{Assaf Naor}
\address{Department of Mathematics, Princeton NJ 08544-1000}
\email{naor@math.princeton.edu}

\thanks{Supported  by NSF grants DMS-2453936 and DMS-2604754, and a Simons Investigator award.  }

\usepackage{xcolor}
\definecolor{maroon}{HTML}{AF3235}

\begin{document} 

\begin{abstract} 
We study a factorization notion for Lipschitz functions between metric spaces in which such a function is written as a composition of a H\"older function and a function that suitably  ``undoes'' the H\"older regularity. We show simple ways to construct such ``de-H\"oldering'' factorizations. If the identity mapping on a metric space $\MM$ admits a de-H\"oldering factorization through a metric space $\cZ$ that has a conical geodesic bicombing,  then the class of metric spaces from  which one can extend $\cZ$-valued Lipschitz functions is shown to be contained  in   the corresponding class of metric spaces for $\MM$-valued Lipschitz functions. As a quick consequence of these abstract permanence properties, we deduce that every $L_1$-valued Lipschitz function from a subset of $\ell_2$ can be extended to a Lipschitz function that takes values in $L_1$ and is defined on all of $\ell_2$, answering a 1992 question of Ball. By work of Makarychev and Makarychev, this implies that every weighted graph has a vertex cut sparsifier of size $n$ and quality $O(\sqrt{\log n})$, improving Moitra's 2009  bound. We also show that for every metric space $\cZ$ that has a conical geodesic bicombing, any metric transform of a metric space $\MM$ has $\cZ$-valued Lipschitz extension modulus at most a universal  constant multiple of the $\cZ$-valued Lipschitz extension modulus of $\MM$ itself,  improving the 2002 bound of Brudnyi and Shvartsman. 
\end{abstract}

\maketitle

\vspace{-0.3in}




\section{Introduction}

Let $(\SS,d_\SS),(\TT,d_\TT)$ be metric spaces. For $\sub\subset \SS$, denote by $\ee(\SS,\sub;\TT)$ the infimum over $K\in [1,\infty]$ such that for every $L\ge 0$ and every $L$-Lipschitz function  $f:\sub\to \TT$ there is a $KL$-Lipschitz function $F:\SS\to \TT$ that extends $f$. When $(\bY;\|\cdot\|_\bY)$ is a Banach space, one also denotes by $\ee_l(\SS,\sub;\bY)$ the analogous parameter with the additional requirement that $F:\SS\to \bY$ depends linearly on $f:\sub\to \bY$. The Lipschitz extension modulus $\ee(\SS;\TT)$ of the above pair of metric spaces is  defined to be the supremum of $\ee(\SS,\sub;\TT)$ over all the nonempty subsets $\sub$ of $\SS$, and one analogously defines $\ee_l(\SS;\bY)$; see Figure~\ref{fig:commuting}.

\begin{figure}[h]
\centering
\fbox{
\begin{minipage}{6.25in}
$$
\xymatrix @C=7.7pc {
\SS \ar@{-->}[rd]^{F} \\
\strut   \sub \ar@{^{(}->}[u]^{\mathsf{Id}_{\sub\to \SS}} \ar[r]^{f} & \TT
}
$$
\caption{\em \small $\ee(\SS;\TT)$ is the infimal $K\in [1,\infty]$ such that for every $\sub\subset \SS$, every $L\ge 0$ and every $L$-Lipschitz function $f:\sub\to \TT$ there exists a $KL$-Lipschitz function $F:\SS\to \TT$ for which the above diagram commutes, where $\mathsf{Id}_{\sub\to \SS}:\sub\to \SS$ is the formal inclusion. }\label{fig:commuting}
\end{minipage}
}
\end{figure}

The present work is part of a development of a research direction that was proposed in~\cite{CNR25} (see specifically Question~26 there and the discussion surrounding it), which we will next briefly recall.  The overarching challenge is to characterize for any source metric space $(\SS,d_\SS)$ the class $T(\SS)$ of all those target metric spaces $(\TT,d_\TT)$ for which $\ee(\SS,\TT)<\infty$. For example, it was asked in~\cite{CNR25} if given two Banach spaces  $\bX,\bY$, the equality  $T(\bX)=T(\bY)$ implies that $\bX$ and $\bY$ are bi-Lipschitz equivalent; this interesting possibility remains open. Also, \cite[Question~26]{CNR25} asks to characterise the class  $T(\ell_2)$.  Ball asked~\cite{Bal92} whether $L_1\in T(\ell_2)$; our investigations herein show that the answer is positive.\footnote{Throughout, $L_p$ stands for the corresponding Lebesgue space on the interval $[0,1]$, equipped with Lebesgue measure. }  One can also naturally consider the ``dual'' class $S(\TT)$ of a metric space $(\TT,d_\TT)$, consisting of all those  source metric spaces $(\SS,d_\SS)$ for which $\ee(\SS,\TT)<\infty$.  We have been developing the aforementioned direction by pursuing the (open ended) project to understand ``algebraic'' properties of $S(\TT)$ and $T(\SS)$, notably which metric space operations preserve those classes. There is a variety of statements that could be obtained in this context, starting with closure under certain product and limiting operations. We postpone such a systematic study to a future publication (it is mostly straightforward, but there is a lot to say), and highlight herein   only one of its components. Our main contribution is Definition~\ref{def:de holder} below, which is a notion of factorization of Lipschitz functions that is ``right'' in the sense that it permits factoring through maps that are available in the nonlinear setting, but not present in the  theory of factorization of linear operators~\cite{Mau75,Pis86,DJT95}; the latter is a major source of inspiration, through the Ribe program~\cite{Bou86,Kal08,Nao12-ribe,Bal13,Ost13,God17,Nao18,BE26}, for research on Lipschitz extension. With Definition~\ref{def:de holder} at hand, the reasoning herein is abstract, structural, and  quite ``soft.'' Nevertheless, it leads to progress on interesting analytic issues, as we will explain below.  

Throughout what follows, all metric spaces are tacitly assumed to be separable and complete (separability can be dropped but assumed for convenience, as this is the setting of all of our applications). Completeness implies that for computing $\ee(\SS;\TT)$ it suffices to consider only closed subsets $\sub$ of $\SS$, as  one can  extend any Lipschitz function to the closure of its domain without increasing the Lipschitz constant. 

\begin{definition}\label{def:de holder} Suppose that $(\MM,d_\MM), (\NN,d_\NN), (\cZ,d_\cZ)$ are metric spaces and $f:\MM\to \NN$. Given  $0<\theta\le 1$ and $K\ge 1$, we say that $f$ admits a $K$-de-$\theta$-H\"oldering factorization through  $\cZ$ if there exist $g:\MM\to \cZ$ and $h:\cZ\to \NN$ with $f=h\circ g$, as well as $\alpha,\beta>0$ with $\alpha\beta\le K$, such that $g$ is $\theta$-H\"older with constant $\alpha$, namely, 
\begin{equation}\label{eq:theta holder}
\forall x,y\in \MM,\qquad d_\cZ\big(g(x),g(y)\big)\le \alpha d_\MM(x,y)^{\theta},
\end{equation}
and $h$ satisfies the following requirement for every $z,w\in \cZ$:
\begin{equation}\label{eq:1/theta}
 d_\NN\big(h(z),h(w)\big)\le \beta^\frac{1}{\theta}\left\{\begin{array}{ll} d_\cZ(z,w)^{\frac{1}{\theta}} &\mathrm{if\ } d_\cZ(z,w)\ge d_\cZ\big(\{z,w\},g(\MM)\big),\\ d_\cZ\big(\{z,w\},g(\MM)\big)^{\frac{1}{\theta}-1}d_\cZ(z,w)&\mathrm{if\ }d_\cZ(z,w)\le d_\cZ\big(\{z,w\},g(\MM)\big).  \end{array}\right.
\end{equation}

We say that $f$ admits a de-H\"oldering factorization through $\cZ$ if there are some $0<\theta\le 1$ and $K\ge 1$ such that $f$ admits a  $K$-de-$\theta$-H\"oldering factorization through  $\cZ$.
\end{definition}

\begin{figure}[h]
\centering
\fbox{
\begin{minipage}{6.25in}
\[
\xymatrix@C=7.7pc@R=3.2pc{
& \cZ \ar@{-->}[dr]^{h} & \\
\MM \ar[rr]^{f} \ar@{-->}[ur]^{g} && \NN
}
\]
\caption{\em\small $f$ admits a de-H\"oldering factorization through $\cZ$ if there exist $g:\MM\to \cZ$ and $h:\cZ\to \NN$ for which the above diagram commutes, as well as $0<\theta\le 1$ and $\alpha,\beta>0$ such that $g$ satisfies the $\theta$-H\"older condition~\eqref{eq:theta holder} and  $h$ satisfies~\eqref{eq:1/theta}, whence it is in particular  $(1/\theta)$-H\"older on pairs of points in $\cZ$ whose distance is large relative to their distance to $g(\MM)\subset \cZ$. }
\label{fig:factorization}
\end{minipage}
}
\end{figure}

Observe that in the setting of Definition~\ref{def:de holder}, by combining the requirement $f=h\circ g$ with~\eqref{eq:theta holder} and the case $z,w\in g(\MM)$ of~\eqref{eq:1/theta}, one sees that $f$ must be Lipschitz; in fact, its Lipschitz constant is at most $(\alpha\beta)^{1/\theta}$. Thus, we are interested in factoring a Lipschitz function using a function $g$ that is $\theta$-H\"older. So, the role of the function $h$ is to ``de-H\"older'' the effect of $g$. For that goal, it is initially tempting to require that  $h$ is H\"older with exponent $1/\theta$, but if $\cZ$  has many rectifiable curves (Definition~\ref{def:de holder} will be used herein for such spaces, notably Banach spaces), then this is a severe restriction, as $h$ is constant on any such curve if $\theta<1$. Thus, \eqref{eq:1/theta} asks for $h$ to  have $(1/\theta)$-H\"older behavior only on pairs $w,z\in \cZ$ whose distance is sufficiently large relative to their distance from the image $g(\MM)$; for the rest of $z,w\in \cZ$, the requirement~\eqref{eq:1/theta} is Lipschitz behavior whose quality may deteriorate as the distance between $\{x,y\}$ and $g(\MM)$ grows. 

The above notion of factorization is well behaved under various operations, examples of which will be demonstrated herein via reasoning that is by and large  straightforward applications of its definition. One such permanence property that is worthwhile noting at this juncture is the following lemma whose proof, which appears in Section~\ref{sec:poitwise}, is a direct point-wise implementation of Definition~\ref{def:de holder}:

\begin{lemma}\label{lem:point-wise} Fix $0<\theta\le 1$ and $K,p\ge 1$, as well as a finite measure space $(\Omega,\mu)$. Let $(\MM,d_\MM), (\NN,d_\NN), (\cZ,d_\cZ)$  be metric spaces. Suppose that  $f:\MM\to \NN$  admits a $K$-de-$\theta$-H\"oldering factorization through  $\cZ$. Define $F:L_p(\mu;\MM)\to L_p(\mu;\NN)$ by setting $F(\psi)=f\circ \psi$ for every $\psi\in L_p(\mu;\MM)$.  Then, $F$ admits a $2K$-de-$\theta$-H\"oldering factorization through  $L_{p/\theta}(\mu;\cZ)$.\footnote{For a metric space $(\MM,d_\MM)$ and a finite measure space $(\Omega,\mu)$, the space $L_p(\mu;\MM)$ consists of all the equivalence classes (modulo almost-everywhere equality) of $\mu$-to-Borel measurable functions $f:\Omega\to \MM$ with $\int_\Omega d_\MM(f(\omega),x_0)^p\ud \mu(\omega)<\infty$ for some  $x_0\in \MM$  (hence, this holds for  all $x_0\in \MM$, by the triangle inequality and the finiteness of $\mu$). It is straightforward to obtain a version of Lemma~\ref{lem:point-wise} for measures that are not finite, but we prefer to omit this because then the spaces do not depend on $x_0$. }
\end{lemma}

Another source of de-H\"oldering factorizations  comes from Proposition~\ref{prop:retract} below, whose proof, which appears in Section~\ref{sec:retract}, relies on a direct adaptation of how McShane classically proved~\cite{McS34} the nonlinear Hahn--Banach theorem; see also~\cite[Chapter~1]{BL00}.  

\begin{prop}\label{prop:retract} Fix $0<\theta\le 1$ and $\otherrho,D\ge 1$. Suppose that $(\MM,d_\MM)$ is an absolute $\otherrho$-Lipschitz retract such that its $\theta$-snowflake $(\MM,d_\MM^\theta)$ embeds with distortion $D$  into some metric space $(\cZ,d_\cZ)$. Then, the identity mapping $\Id_\MM:\MM\to \MM$ admits a $O( D\r^\theta)$-de-$\theta$-H\"oldering factorization through  $\cZ$.
\end{prop}

For Proposition~\ref{prop:retract}, one says that a metric space $(\MM,d_\MM)$ embeds with distortion $D\ge 1$ into a metric space $(\cZ,d_\cZ)$  if there exist $f:\MM\to \cZ$ and $\upsigma>0$ such that $\upsigma d_\MM(x,y)\le d_\cZ(f(x),f(y))\le D\upsigma d_\MM(x,y)$ for every $x,y\in \MM$. A metric space $(\MM,d_\MM)$ is a $\r$-absolute Lipschitz retract for $\r\ge 1$ if for every metric space $(\SS,d_\SS)$ and an isometric embedding $j:\MM\to \SS$,  there is a $\r$-Lipschitz retraction from $\SS$ onto $j(\MM)$. Absolute Lipschitz retracts are a substantial class of metric spaces; see e.g.~\cite{AP56,Lin64,PY95,LS05,Lan13}, as well as the exposition in~\cite[Chapter~1]{BL00}.  The real line (in fact, any complete metric tree~\cite{Dre84}) is a $1$-absolute Lipschitz retract. By~\cite{Sch37} the metric space $(\R,|\cdot-\cdot|^\theta)$ embeds isometrically into $\ell_2$ for any $0<\theta\le 1$ (and, by~\cite[Proposition~3 and Proposition~9]{NR26} the metric space $(\R,|\cdot-\cdot|^\theta)$ embeds with distortion $O(1)$ into $\X^3$ for any normed space $(\X,\|\cdot\|_\X)$ satisfying $\dim \X\ge C/\theta$ for a suitable universal constant $C\ge 1$). As $\ell_2$ embeds isometrically into $L_q$ for every $1\le q\le \infty$,  also $(\R,|\cdot-\cdot|^\theta)$ embeds isometrically into $L_q$.  An application of Proposition~\ref{prop:retract} now shows that for every $1\le p\le q$ the identity mapping on $\R$  admits a $O(1)$-de-$(p/q)$-H\"oldering factorization through $L_q$, whence by Lemma~\ref{lem:point-wise} the identity mapping on $L_p$ admits a $O(1)$-de-$(p/q)$-H\"oldering factorization through $L_q(L_q)\cong L_q$. We thus get:

\begin{corollary}\label{coro:Lp case} If $1\le p\le q<\infty$, then $\Id_{L_p}$ admits a $O(1)$-de-$\frac{p}{q}$-H\"oldering factorization through $L_q$.
\end{corollary}

Proposition~\ref{prop:def to ext}  below is a  link between Definition~\ref{def:de holder} and the aforementioned discussion on permanence properties of the metric space class $S(\cdot)$. It uses the notion of metric spaces that have a conical geodesic bicombing, which holds for all Banach spaces and all nonpositively curved spaces in the sense of Busemann  or Alexandrov, so for those who do not wish to work with this abstraction (even though it was beneficial for us to do so for a reason that is discussed below), it is sufficiently meaningful  to consider only Banach spaces (this covers all of the concrete applications herein). 

Following the early investigations~\cite{Bus48,Bus55,Ito79,Bus87,ECHLPT92} and the thorough recent treatments in e.g.~\cite{Lan13,DL15,Bas18}, a conical geodesic bicombing on a metric space $(\cZ,d_\cZ)$ is a function 
\begin{equation}\label{eq:bicombing1}
\gamma:\cZ\times \cZ\times [0,1]\to \cZ,
\end{equation}
 such that for all $x,y\in \cZ$  the path $(t\in [0,1])\mapsto \gamma(x,y,t)$ is a constant speed geodesic joining  $x$ to $y$, namely, 
 \begin{equation}\label{eq:bicombing2}
 \gamma(x,y,0)=x\quad\mathrm{and} \qquad \gamma(x,y,1)=y\qquad \mathrm{and}\qquad  d_\cZ\big(\gamma(x,y,s), \gamma(x,y,t)\big)=|s-t|d_\cZ(x,y),
 \end{equation}
  for every $0\le s,t\le 1$, and such that  the following convexity requirement holds: 
\begin{equation}\label{eq:bicombing3}
\forall x,y,z,w\in \cZ,\ \forall 0\le t\le 1,\qquad d_\cZ\big(\gamma(x,y,t),\gamma(z,w,t)\big)\le (1-t)d_\cZ(x,z)+td_\cZ(y,w).
\end{equation}
Any Banach space $(\bZ,\|\cdot\|_\bZ)$ has a conical geodesic bicombing, as seen by taking $\gamma(x,y,t)=(1-t)x+ty$ for every $x,y\in \bZ$ and $0\le t\le 1$.

\begin{prop}\label{prop:def to ext} Let $(\MM,d_\MM), (\cZ,d_\cZ)$ be metric spaces such that $\cZ$ has a conical geodesic bicombing. If the identity mapping  on $\MM$  admits a de-H\"oldering factorization through $\cZ$, then   $S(\cZ)\subset S(\MM)$. 

Quantitatively, if the identity mapping on $\MM$ admits  a $K$-de-$\theta$-H\"oldering factorization through  $\cZ$ for some $0<\theta\le 1$ and $K\ge 1$, then $\ee(\SS,\MM)^\theta\lesssim K\ee(\SS,\cZ)$ for every metric space $(\SS,d_\SS)$.\footnote{We use the following  asymptotic notation, in addition to  the usual $O(\cdot)$ notation. Given $a,b>0$, by writing
$a\lesssim b$ or $b\gtrsim a$ we mean that $a\le \kappa b$ for some
universal constant $\kappa>0$, and $a\asymp b$
stands for $(a\lesssim b) \wedge  (b\lesssim a)$. }
\end{prop}

Proposition~\ref{prop:def to ext} is proved in Section~\ref{sec:deH to ext}. Its elementary reasoning combines steps  in~\cite{Nao01,LN05,Nao24}, so the main contribution herein is the mere realization of the relevance of Definition~\ref{def:de holder}. A secondary conceptual contribution of how we prove Proposition~\ref{prop:def to ext},  is the abstraction of the aforementioned steps.  We performed that abstraction since~\cite{Nao01} relied on an assumption that at the time plausibly held for arbitrary Banach spaces (and, that assumption easily holds in our main applications herein), but subsequently counterexamples to it were found in~\cite{Kal11,Acu26}. Therefore, we were curious to see how to generalize~\cite{Nao01}, and this forced us to find a simpler route that, as a byproduct, settles the following  issue that was left open by~\cite{BS02} and was raised explicitly in e.g.~\cite[Remark~141]{Nao24}.

If $\omega:[0,\infty)\to [0,\infty)$ is nonconstant, nondecreasing, concave, and with $\omega(0)=0$, then for every metric space $(\MM,d_\MM)$ also $(\MM,\omega\circ d_\MM)$ is a metric space,  denoted $\omega\circ \MM$ when the metric $d_\MM$ is clear from the context. It was proved in~\cite{BS02} that $\ee(\omega\circ \MM;\bZ)\lesssim \ee(\MM;\bZ)^2$ and $\ee_l(\omega\circ \MM;\bZ)\lesssim \ee_l(\MM;\bZ)^2$ for every Banach space $(\bZ,\|\cdot\|_\bZ)$. The H\"older case $\omega(t)=t^\theta$ for some $0<\theta<1$ was previously treated in~\cite{Nao01}, where the same quadratic loss is incurred under a mild assumption on $\bZ$. Because the proofs in~\cite{Nao01} and~\cite{BS02} differ substantially from each other  (the reasoning in~\cite{Nao01} is elementary while the proof in~\cite{BS02} interestingly relies on the $K$-divisibility theorem~\cite{BK91}  from interpolation theory), one naturally wonders if the aforementioned  quadratic dependence on $\ee(\MM;\bZ)$ is necessary; the following theorem is a byproduct of our proof of Proposition~\ref{prop:def to ext} , showing that, in fact, a linear estimate holds, in greater generality: 

\begin{theorem}\label{thm:linear BS} Let $(\MM,d_\MM), (\cZ,d_\cZ)$ be metric spaces such that $\cZ$ has a conical geodesic bicombing. Then, $\ee(\omega\circ \MM,\cZ)\lesssim \ee(\MM,\cZ)$ for every $\omega:[0,\infty)\to [0,\infty)$ that is not identically $0$, nondecreasing, continuous,  concave, and satisfies  $\omega(0)=0$. Furthermore, if  $(\bZ,\|\cdot\|_\bZ)$ is a Banach space, then  $\ee_l(\omega\circ \MM,\bZ)\lesssim \ee_l(\MM,\bZ)$.
\end{theorem}

The following corollary is a substitution of Corollary~\ref{coro:Lp case}  into Proposition~\ref{prop:def to ext}: 

\begin{corollary}\label{cor:extension monotonicity} If $1\le p\le q$, then $S(L_q)\subset S(L_p)$.  In fact, $\ee(\SS;L_p)^{p/q}\lesssim \ee(\SS;L_q)$ for every metric space $(\SS,d_\SS)$.
\end{corollary}

As $\ell_2\in S(\ell_2)$ by~\cite{Kir34}, it follows from Corollary~\ref{cor:extension monotonicity} that $\ell_2\in S(L_1)$, i.e., $\ee(\ell_2;L_1)<\infty$, which is a positive answer to Ball's extension problem~\cite{Bal92}. More generally:

\begin{corollary}\label{thm:endpoint NPSS} For every $1\le p\le 2\le q<\infty$ we have $\ee(L_q;L_p)\lesssim \sqrt{q\min\{\log q, 1/(p-1)\}}\le\sqrt{q\log q}$.
\end{corollary}
Prior to Corollary~\ref{thm:endpoint NPSS}, it was known~\cite{NPSS06}  that $\ee(L_q;L_p)\lesssim \sqrt{q/(p-1)}$ for every $1<p\le 2\le q<\infty$. The new information  herein is thus that there is no divergence as $p\to 1^+$, though we suspect that the $\log q$ term in  Corollary~\ref{thm:endpoint NPSS} could be removed, i.e., that $\ee(L_q;L_p)\lesssim \sqrt{q}$ for every $1\le p\le 2\le q<\infty$. 

To explain how  Corollary~\ref{thm:endpoint NPSS}  follows from Corollary~\ref{cor:extension monotonicity},  fix $p\le r\le 2$ and apply Corollary~\ref{cor:extension monotonicity} with $\SS=L_q$ to get $\ee(L_q;L_p)\lesssim \ee(L_q;L_r)^{r/p}$. By~\cite{NPSS06} (which relies on Ball's deep extension theorem~\cite{Bal92}) we have  $\ee(L_q;L_r)\lesssim \sqrt{q/(r-1)}$. So, $\ee(L_q;L_p)\lesssim (q/(r-1))^{r/(2p)}$. Corollary~\ref{thm:endpoint NPSS} follows by optimizing over $p\le r\le 2$.

\begin{remark}\label{rem:Mcotype} {\em By a beautiful construction of Kalton~\cite{Kal12} there is a closed linear subspace $\bZ$ of $L_1$ satisfying $\ee(\ell_2;\bZ)=\infty$. We get $\ee(\ell_2;L_1)<\infty$  using  Lemma~\ref{lem:point-wise}, which is quite trivial, but it provides  a de-H\"oldering factorization of the identity on all of $L_1$. Ball's extension theorem~\cite{Bal92} is a different approach to Lipschitz extension which influentially utilizes metric invariants that are inspired by Maurey's extension theorem~\cite{Mau74} for linear operators. It remains an interesting question to determine if the pertinent invariant, namely, metric Markov cotype $2$,  holds for $L_1$; see also the discussion in~\cite{MN13}. We did not yet investigate if the simple method herein could be used to show that $L_1$ has metric Markov cotype $2$, but we will  do that in future work as it would be valuable to recast $\ee(\ell_2;L_1)<\infty$ as part of Ball's theory. } 
\end{remark}

We end this introduction by describing an application of $\ee(\ell_2;L_1)<\infty$ to combinatorics and computer science, which is a black box consequence of $\ee(\ell_2;L_1)<\infty$ thanks to an insightful reduction~\cite{MM16} by Makarychev and Makarychev. For each integer $n\ge 3$, let  $Q^{\mathrm{cut}}_n$ be the infimum over those $Q\ge 1$ such that for every finite weighted graph $(V,w)$, namely, $V$ is a finite set and $w:V\times V\to [0,\infty)$ is a symmetric function, and every $U\subset V$ with $|U|\le n$ there exists a symmetric function $w_U:U\times U\to [0,\infty)$  such that:
\begin{equation}\label{eq:Q quality}
\forall S\subset U,\qquad \min_{\substack{T\subset V\\ T\cap U=S}} w\big(T\times (V\setminus T)\big)\le w_U\big(S\times (U\setminus S)\big)\le Q \min_{\substack{T\subset V\\ T\cap U=S}} w\big(T\times (V\setminus T)\big),
\end{equation}
where we set $w(E)=\sum_{(u,v)\in E} w(u,v)$ for every $E\subset V\times V$. The interpretation of~\eqref{eq:Q quality} is that $w_U$ induces a weighted graph on the vertex set $U$ with the property that for { every}  $S\subset U$ its edge boundary in  $(U,w_U)$ is up to factor $Q$ the  minimum edge boundary of  a cut separating $S$ and $U\setminus S$ in the (potentially much larger) original graph $(V,w)$. If $n$ and $Q$ are small, then this is a valuable tool for reducing the dependence on $n$ in combinatorial optimization problems~\cite{Moi09,LM10,CLLM10,EGKRTT14}. 

The above sparsification notion is due to Moitra~\cite{Moi09}, who proved $Q_n^{\mathrm{cut}}\lesssim (\log n)/\log\log n$, which remained best known despite substantial efforts over the years. By combining~\cite{MM16} with~\cite{CNR25}, we have $Q_n^{\mathrm{cut}} \lesssim \ee(\ell_2;L_1)\sqrt{\log n}$; see equation~(38) in~\cite{CNR25} and the discussion surrounding it. We thus deduce that $Q_n^{\mathrm{cut}} \lesssim \sqrt{\log n}$, which almost matches the best known lower bound $Q_n^{\mathrm{cut}} \gtrsim \sqrt{\log n}/\log\log n$ that was proved in~\cite{MM16}, relying  on~\cite{FJS88}.

\section{Proof of Lemma~\ref{lem:point-wise}}\label{sec:poitwise}  

Fix $g:\MM\to \cZ$ and $h:\cZ\to \NN$ with $h\circ g=f$ such that~\eqref{eq:theta holder} and~\eqref{eq:1/theta} hold for some $\alpha,\beta>0$ satisfying $\alpha\beta\le K$. For every $q>0$, by applying~\eqref{eq:theta holder}  point-wise we get that every $\f,\psi\in L_{q\theta}(\mu;\MM)$ satisfy:
\begin{equation}\label{eq:lifted G}
 d_{L_{q}(\mu;\cZ)}(g\circ \f,g\circ \psi) \le \alpha d_{L_{q\theta}(\mu;\MM)}( \f, \psi)^\theta.  
\end{equation}
Also, a point-wise application of~\eqref{eq:1/theta} gives the following inequality for every $\xi,\zeta\in L_q(\mu;\cZ)$:
\begin{align}\label{eq:de holder point wise}
\begin{split}
d_{L_{q\theta}(\mu;\NN)}(h\circ \xi,h\circ \zeta)^\theta&\le \beta\Big(\int_\Omega \max \big\{d_\cZ(\xi,\zeta),  d_\cZ\big(\{\xi,\zeta\},g(\MM)\big)^{1-\theta}d_\cZ(\xi,\zeta)^\theta\big\}^q\ud \mu\Big)^{\frac{1}{q}}\\&
\le 
\beta \Big(d_{L_q(\mu;\cZ)}(\xi,\zeta)^q+ \int_{\Omega} d_\cZ\big(\{\xi,\zeta\},g(\MM)\big)^{q(1-\theta)}d_\cZ(\xi,\zeta)^{q\theta}\ud \mu\Big)^{\frac{1}{q}}\\
&\le 2^{\frac{1}{q}} \beta\max \bigg\{d_{L_q(\mu;\cZ)}(\xi,\zeta),\Big(\int_{\Omega} d_\cZ\big(\{\xi,\zeta\},g(\MM)\big)^{q}\ud
 \mu\Big)^{\frac{1-\theta}{q}}d_{L_q(\mu;\cZ)}(\xi,\zeta)^\theta\bigg\},
 \end{split}
\end{align}
where the final step of~\eqref{eq:de holder point wise} uses H\"older's inequality. Observe that for every $\phi\in L_{q}(\mu;\cZ)$ we (trivially) have:
$$
d_{L_q(\mu;\cZ)}\big(\phi,g\circ L_{q\theta}(\mu;\MM)\big)\ge \Big(\int_{\Omega} d_\cZ\big(\phi,g(\MM)\big)^{q}\ud
 \mu\Big)^{\frac{1}{q}}.
$$
Consequently,
 $$
 \Big(\int_{\Omega} d_\cZ\big(\{\xi,\zeta\},g(\MM)\big)^{q}\ud
 \mu\Big)^{\frac{1}{q}}\le d_{L_q(\mu;\cZ)}\big(\{\xi,\zeta\},g\circ L_{q\theta}(\mu;\MM)\big),
$$
which together with~\eqref{eq:de holder point wise} gives:
\begin{align}\label{eq:after holder}
d_{L_{q\theta}(\mu;\NN)}(h\circ \xi,h\circ \zeta)^\theta\le 
 2^{\frac{1}{q}}\beta\max \bigg\{d_{L_q(\mu;\cZ)}(\xi,\zeta),d_{L_q(\mu;\cZ)}\big(\{\xi,\zeta\},g\circ L_{q\theta}(\mu;\MM)\big)^{1-\theta}d_{L_q(\mu;\cZ)}(\xi,\zeta)^\theta\bigg\}.
\end{align}

Define $G: L_p(\mu;\MM)\to L_{p/\theta}(\mu;\cZ)$ by  $G(\f)=g\circ \f$ for  $\f\in L_p(\mu;\MM)$.  By~\eqref{eq:lifted G} for $q=p/\theta$ and $\psi$ a constant function,  $G$ indeed takes values in $L_{p/\theta}(\mu;\cZ)$. Also, \eqref{eq:lifted G}  shows that $G$ is $\theta$-H\"older with constant $\alpha$. Similarly, define $H:L_{p/\theta}(\mu;\cZ)\to L_p(\mu;\NN)$ by $H(\xi)=h\circ \xi$ for $\xi\in L_{p/\theta}(\mu;\cZ)$, so by~\eqref{eq:after holder}  $H$ indeed takes values in   $L_p(\mu;\NN)$ and satisfies the following  inequality for every $\xi,\zeta\in L_{p/\theta}(\mu;\cZ)$:
$$
d_{L_{p}(\mu;\NN)}\big(H(\xi),H(\zeta)\big)^\theta\le 
2^{\frac{\theta}{p}} \beta\max \bigg\{d_{L_{\frac{p}{\theta}}(\mu;\cZ)}(\xi,\zeta),d_{L_{\frac{p}{\theta}}(\mu;\cZ)}\Big(\{\xi,\zeta\},G\big( L_{p}(\mu;\MM)\big)\Big)^{1-\theta}d_{L_{\frac{p}{\theta}}(\mu;\cZ)}(\xi,\zeta)^\theta\bigg\}.
$$
Because $F=H\circ G$, where $F$ is as in the statement of Lemma~\ref{lem:point-wise}, and $2^{\theta/p}\le 2$, the proof is complete.  \qed

\section{Proof of Proposition~\ref{prop:retract}}\label{sec:retract}

The proof of the following lemma consists of substituting a ``super H\"older'' assumption into the infimum convolution that McShane used~\cite{McS34} to prove the nonlinear Hahn--Banach theorem.

\begin{lemma}\label{lem:mcshane variant} Fix $q\ge 1$ and $L>0$, a metric space $(\MM,d_\MM)$, and $\emptyset\neq \sub\subset \MM$. Suppose that $f:\sub\to \R$ satisfies: 
\begin{equation}\label{eq:q holder}
\forall x,y\in \sub,\qquad |f(x)-f(y)|\le Ld_\MM(x,y)^q. 
\end{equation}
Then, there exists $F:\MM\to \R$ that extends $f$ and satisfies:
\begin{equation}\label{eq:mcshane conclusion}
\forall x,y\in \MM,\qquad |F(x)-F(y)|\le e^{O(q)}L \max\big\{d_\MM(x,y)^{q},d_\MM(\{x,y\},\sub)^{q-1}d_\MM(x,y)\big\}.
\end{equation} 
\end{lemma}

Assuming Lemma~\ref{lem:mcshane variant} for the moment, the proof of Proposition~\ref{prop:retract} becomes a straightforward unravelling of the definitions and terminology:

\begin{proof}[Deduction of Proposition~\ref{prop:retract} from Lemma~\ref{lem:mcshane variant}] By assumption, we may fix  $g:\MM\to \cZ$ and $\upsigma>0$ such that: 
$$\forall x,y\in \MM,\qquad \upsigma d_\MM(x,y)^\theta\le d_\cZ(g(x),g(y))\le \upsigma Dd_\MM(x,y)^\theta.$$ 
So, $g$ satisfies the first requirement~\eqref{eq:theta holder} of Definition~\ref{def:de holder} with $\alpha=\upsigma D$. 

To obtain the corresponding de-H\"oldering mapping $h:\cZ\to \MM$ from Definition~\ref{def:de holder}, fix an isometric embedding $\cj:\MM\to \ell_\infty$. Then, $\cj\circ g^{-1}:g(\MM)\to \ell_\infty$ satisfies $\|\mathcal{j}\circ g^{-1}(z)-\mathcal{j}\circ g^{-1} (w)\|_\infty\le (d_\cZ(z,w)/\upsigma)^{1/\theta}$ for every $z,w\in g(\MM)\subset \cZ$, i.e., each of the coordinates of $\mathcal{j}\circ g^{-1}$ satisfies the assumption~\eqref{eq:q holder} of Lemma~\ref{eq:q holder}  for $\sub=g(\MM)$ with $q=1/\theta$ and $L=1/\upsigma^{1/\theta}$. By applying Lemma~\ref{eq:q holder} separately to each of the  coordinates of $\mathcal{j}\circ g^{-1}$, we obtain  $\Phi: \cZ\to \ell_\infty$ satisfying $\Phi|_{g(\MM)}=\cj\circ g^{-1}$, and for every $z,w\in \cZ$: 
$$
 \|\Phi(z)-\Phi(w)\|_\infty \le \Big(\frac{O(1)}{\upsigma}\Big)^{\frac{1}{\theta}} \left\{\begin{array}{ll} d_\cZ(z,w)^{\frac{1}{\theta}} &\mathrm{if\ } d_\cZ(z,w)\ge d_\cZ\big(g(\MM),\{z,w\}\big),\\ d_\cZ\big(g(\MM),\{z,w\}\big)^{\frac{1}{\theta}-1}d_\cZ(z,w)&\mathrm{if\ }d_\cZ(z,w)\le d_\cZ\big(g(\MM),\{z,w\}\big).  \end{array}\right.
$$

As $\MM$ is a $\r$-absolute Lipschitz retract we can fix a retraction $R:\ell_\infty \to \cj(\MM)$ from $\ell_\infty$ onto $\cj(\MM)$ that is $\r$-Lipschitz. We may now define $h= \cj^{-1}\circ R\circ \Phi:\cZ\to \MM$, which is allowed since $R$ takes values in $\cj(\MM)$. Also,  $h\circ g=\Id_\MM$, because $R|_{g(\MM)}=\Id_{g(\MM)}$    as $R$ is a retraction, and $\Phi|_{g(\MM)}=\cj\circ g^{-1}$. The desired inequality~\eqref{eq:1/theta} with $\NN=\MM$ of  Definition~\ref{def:de holder} holds with $\beta\lesssim \r^\theta/\upsigma$, whence $\alpha\beta\lesssim  D\r^\theta$, as required. 
\end{proof}

\begin{proof}[Proof of Lemma~\ref{lem:mcshane variant}] By considering $f/L$ we may assume from now that $L=1$. As $f$ extends to the closure of $\sub$ while still satisfying~\eqref{eq:q holder}, we may also assume from now that $\sub$ is closed.  Define:
\begin{equation}\label{eq:def q mcshane}
\forall x\in \MM,\qquad F(x)\eqdef \inf_{s\in \sub} \big(f(s)+2^{q-1}d_\MM(x,s)^q\big). 
\end{equation}
This coincides with the classical McShane extension formula~\cite{McS34} when $q=1$. For $q>1$, one can replace the  $2^{q-1}$ in~\eqref{eq:def q mcshane} by any parameter $K>1$, but $K=1$ does not work here, unlike when $q=1$; see Remark~\ref{rem:star} below.    The optimal choice of $K$ for the ensuing reasoning is not $2^{q-1}$; we chose the value $2^{q-1}$ for simplicity, as this only influences the implicit constant factor of the $O(q)$ in~\eqref{eq:mcshane conclusion}. 

Fix $s_0\in \sub$. For every $s\in \sub$ and $x\in \MM$ we have: 
\begin{align}\label{eq:not b-infty}
\begin{split}
f(s)+2^{q-1}d_\MM(x,s)^q&\stackrel{\eqref{eq:q holder}}{\ge} f(s_0)-d_\MM(s,s_0)^q+2^{q-1}d_\MM(x,s)^q\\&\ge f(s_0)-\big(d_\MM(s,x)+d_\MM(x,s_0)\big)^q+2^{q-1}d_\MM(x,s)^q\ge f(s_0)-2^{q-1}d_\MM(x,s_0)^q,
\end{split}
\end{align}
where in the final step we used that $(u+v)^q\le 2^{q-1}(u^q+v^q)$ holds for every $u,v\in [0,\infty)$. Therefore, the infimum that defines $F(x)$ is over a set that is bounded from below, whence $F(x)$ is indeed real-valued. 

To check that $F$ extends $f$, suppose that $x\in \sub$. Choosing $s=x$ in the right hand side of~\eqref{eq:def q mcshane} shows  that $F(x)\le f(x)$, while for every $s\in \sub$ we have: 
\begin{equation}\label{eq:F extends}
f(x)\stackrel{\eqref{eq:q holder}}{\le} f(s)+d_\MM(x,s)^q\le f(s)+2^{q-1}d_\MM(x,s)^q.
\end{equation}

So, the desired inequality~\eqref{eq:mcshane conclusion} holds if $x,y\in \sub$, as it is weaker than~\eqref{eq:q holder}. We will therefore next proceed to verify~\eqref{eq:mcshane conclusion} under the assumption $\{x,y\}\not\subseteq \sub$, say, $x\in \MM\setminus \sub$. Thus $d_\MM(x,\sub)>0$, as $\sub$ is closed. We will also assume from now that $q>1$, since $q=1$ is the classical (and simple) analysis in~\cite{McS34} of~\eqref{eq:def q mcshane}.  

Fix  $\{s_n=s_n(x)\}_{n=1}^\infty \subset \sub$ that satisfy:
\begin{equation}\label{eq:lim analysis}
F(x)=\lim_{n\to \infty} \big(f(s_n)+2^{q-1}d_\MM(x,s_n)^q\big),
\end{equation}
For every $s\in \sub$ and $n\in \N$ we have:
\begin{align*}
F(x)\stackrel{\eqref{eq:def q mcshane} }{\le} f(s)+2^{q-1}d_\MM(x,s)^q&\stackrel{\eqref{eq:q holder}}{\le} f(s_n)+d_\MM(s,s_n)^q+2^{q-1}d_\MM(x,s)^q\\&\le f(s_n)+ \big(d_\MM(x,s_n)+d_\MM(x,s)\big)^q+2^{q-1}d_\MM(x,s)^q. 
\end{align*}
By taking the infimum over $s\in \sub$ we get:
\begin{equation}\label{eq:for liminf}
\forall n\in \N,\qquad F(x)-f(s_n)\le \big(d_\MM(x,s_n)+d_\MM(x,\sub)\big)^q+2^{q-1}d_\MM(x,\sub)^q.
\end{equation}
Consequently,
\begin{align*}
2^{q-1}\liminf_{n\to\infty}  d_\MM(x,s_n)^q&\stackrel{\eqref{eq:lim analysis}}{\le} \liminf_{n\to\infty} \big(F(x)-f(s_n)\big)\stackrel{\eqref{eq:for liminf}}{\le} \liminf_{n\to\infty} \Big( \big(d_\MM(x,s_n)+d_\MM(x,\sub)\big)^q+2^{q-1}d_\MM(x,\sub)^q\Big)\\&=\big(\liminf_{n\to\infty} d_\MM(x,s_n)+d_\MM(x,\sub)\big)^q+2^{q-1}d_\MM(x,\sub)^q.
\end{align*}
By~\eqref{eq:lim analysis} and~\eqref{eq:for liminf}, and using $2^{q-1}>1$, we see  that the sequence  $\{d_\MM(x,s_n)\}_{n=1}^\infty$ is bounded, so if we denote $\eta=\liminf_{n\to\infty} d_\MM(x,s_n)/d_\MM(x,\sub)$, then $\eta<\infty$ and we obtained the  inequality  $2^{q-1}\eta^q\le (\eta+1)^q+2^{q-1}$. It follows that $\eta\le \eta_q$, where $\eta_q$ is the unique positive number satisfying $2^{q-1}\eta_q^q= (\eta_q+1)^q+2^{q-1}$. It is a simple calculus exercise to check that $\eta_q\lesssim q/(q-1)$. We have thus shown that:
\begin{equation}\label{coercivce bound}
\liminf_{n\to\infty} d_\MM(x,s_n)\lesssim \frac{q}{q-1}d_\MM(x,\sub). 
\end{equation}

Because $s_n\in \sub$ for every $n\in \N$, every $y\in \MM$ satisfies:
\begin{align}\label{eq:Fy-Fx}
\begin{split}
F(y)-F(x)&\stackrel{\eqref{eq:lim analysis}}{=}F(y)-\lim_{n\to \infty}\big(f(s_n)+2^{q-1}d_\MM(x,s_n)^q\big)\\&\stackrel{\eqref{eq:def q mcshane}}{\le} 2^{q-1}\liminf_{n\to\infty} \big(d_\MM(y,s_n)^q-d_\MM(x,s_n)^q\big)\\&\ \le 2^{q-1}\liminf_{n\to\infty}\Big(\big(d_\MM(x,s_n)+d_\MM(x,y)\big)^q-d_\MM(x,s_n)^q\Big)\\&\ \le2^{q-1}q d_\MM(x,y)\liminf_{n\to\infty}
\big(d_\MM(x,s_n)+d_\MM(x,y)\big)^{q-1}\\ 
&\stackrel{\eqref{coercivce bound}}{\lesssim} e^{O(q)} d_\MM(x,y)\big(d_\MM(x,\sub)+d_\MM(x,y)\big)^{q-1}\\
&\ \lesssim e^{O(q)} d_\MM(x,y)\big(d_\MM(\{x,y\},\sub)+d_\MM(x,y)\big)^{q-1}\\
&\ \lesssim e^{O(q)}\max\big\{d_\MM(x,y)^{q},d_\MM(\{x,y\},\sub)^{q-1}d_\MM(x,y)\big\},
\end{split}
\end{align}
where the fourth step of~\eqref{eq:Fy-Fx} is a straightforward application of the mean value theorem to the function $(u\ge 0)\mapsto u^q$, and the penultimate step of~\eqref{eq:Fy-Fx} uses $d_\MM(x,\sub)\le d_\MM(y,\sub)+d_\MM(x,y)$. As we proved~\eqref{eq:Fy-Fx} for $x\in \MM\setminus \sub$ and $y\in \MM$, if also $y\in \MM\setminus \sub$, then its application with the roles of $x,y$ interchanged gives: 
\begin{equation}\label{eq:almost done}
|F(x)-F(y)|\le e^{O(q)}\max\big\{d_\MM(x,y)^{q},d_\MM(\{x,y\},\sub)^{q-1}d_\MM(x,y)\big\}.
\end{equation}
If $y\in \sub$, then $d_\MM(\{x,y\},\sub)=0$ and  $F(x)-F(y)=F(x)-f(y)\le 2^{q-1} d_\MM(x,y)^q$ by~\eqref{eq:def q mcshane} with $s=y$, so  together with~\eqref{eq:Fy-Fx}  the conclusion~\eqref{eq:almost done} holds when $y\in \sub$ as well. This completes the proof of~\eqref{eq:mcshane conclusion}.
\end{proof}

\begin{remark}\label{rem:star} {\em In the context of Lemma~\ref{lem:mcshane variant} while continuing to assume the normalization $L=1$, one can take any $K>1$ and work with the following extension formula in place of~\eqref{eq:def q mcshane}:
\begin{equation}\label{eq:def q mcshane K}
\forall x\in \MM,\qquad F_K(x)\eqdef \inf_{s\in \sub} \big(f(s)+Kd_\MM(x,s)^q\big). 
\end{equation}
An inspection of the proof of Lemma~\ref{lem:mcshane variant} reveals that it could be carried out for any $K>1$, yielding the same inequality as~\eqref{eq:mcshane conclusion} for $F_K$, except that the $e^{O(q)}$ factor is now replaced by  $e^{O(q)} K/(1-K^{-1/q})^{q-1}$, and a more careful analysis yields a more complicated expression which can be optimized over $K>1$ to yield~\eqref{eq:mcshane conclusion} with $e^{O(q)}$ replaced by  $O(q6^q)$; the (quite tedious) details of this are omitted. Thus, our above choice $K=2^{q-1}$ is not best possible,  but it does not materially change~\eqref{eq:mcshane conclusion}. 

Suppose that $q>1$. Note that $F_k$ as in~\eqref{eq:def q mcshane K} extends $f$ for any $K\ge 1$, by reasoning exactly as in~\eqref{eq:F extends}. Any $K>1$ ensures that $F_K(x)$ cannot equal $-\infty$, as seen by the following variant of~\eqref{eq:not b-infty}: 
\begin{align*}
f(s)+Kd_\MM(x,s)^q&\stackrel{\eqref{eq:q holder}}{\ge} f(s_0)-d_\MM(s,s_0)^q+Kd_\MM(x,s)^q\\&\ge f(s_0)-\big(d_\MM(s,x)+d_\MM(x,s_0)\big)^q+Kd_\MM(x,s)^q\ge f(s_0)-\frac{K}{\left(K^{\frac{1}{q-1}}-1\right)^{q-1}}d_\MM(x,s_0)^q,
\end{align*}
where the last step uses the following  application of H\"older's inequality:
$$
\forall u,v\ge 0,\qquad (u+v)^q \le Ku^q+\frac{K}{\left(K^{\frac{1}{q-1}}-1\right)^{q-1}}v^q.
$$
However, taking $K=1$ in~\eqref{eq:def q mcshane K}  could result in $F_1(x)=-\infty$. Indeed, consider $\MM=\{x,y,z_1,z_2,\ldots\}$, 
equipped with the star metric: 
$$
d_\MM(x,y)=1,\quad \mathrm{and}\quad \forall n\in \N,\quad d_\MM(x,z_n)=n=d_\MM(y,z_n)-1,\quad \mathrm{and}\quad \forall k,\ell\in \N,\quad d_\MM(z_k,z_\ell)=(k+\ell)\d_{k\ell}.  
$$
If we take $\sub=\{y,z_1,z_2,\ldots\}$ and define $f:\sub \to \R$ by $f(y)=0$ and $f(z_n)=-(n+1)^q$, then it is straightforward to check that $|f(s)-f(t)|\le d_\MM(s,t)^q$   for every $s,t\in \sub$. But, because $q>1$ we have:
$$
\forall n\in \N,\qquad F_1(x)\le f(z_n)+d_\MM(x,z_n)^q=-(n+1)^q+n^q\xrightarrow[n \to \infty]{}-\infty.
$$   }
\end{remark}

\section{Proof of Proposition~\ref{prop:def to ext} and Theorem~\ref{thm:linear BS} }\label{sec:deH to ext}

We will prove the following general lemma, whose proof is essentially a concatenation of steps that were carried out in~\cite{Nao01,LN05,Nao24} (for similar purposes):

\begin{lemma}\label{lem:two scales} Let $\omega:(0,\infty)\to (0,\infty)$ be nondecreasing and concave, with $\lim_{t\to 0^+}\omega(t)=0$. Fix  metric spaces $(\MM,d_\MM), (\cZ,d_\cZ)$ such that $\cZ$ has a conical geodesic bicombing. If $\sub\subset \MM$ is closed and $f:\sub\to \cZ$ satisfies:
\begin{equation}\label{eq:omega modulus}
\forall x,y\in \sub,\qquad d_\cZ\big(f(x),f(y)\big)\le \omega \big(d_\MM(x,y)\big),
\end{equation}
then there exists $F:\MM\to \cZ$ such that $F(a)=f(a)$ for every $a\in \sub$,  and for every distinct $x,y\in \MM$ we have:
\begin{equation}\label{eq:omega extended}
d_\cZ\big(F(x),F(y)\big)\lesssim \ee(\MM;\cZ)\frac{\omega\left(d_\MM(x,y)+d_\MM(\{x,y\},\sub)\right)}{d_\MM(x,y)+d_\MM(\{x,y\},\sub)}d_\MM(x,y).
\end{equation}
Furthermore, if $\cZ$ is a Banach space, then $F$ can be taken to depend linearly on $f$, with $\ee(\MM;\cZ)$ replaced by $\ee_l(\MM;\cZ)$ in~\eqref{eq:omega extended}.  
\end{lemma}

Theorem~\ref{thm:linear BS} is a consequence of Lemma~\ref{lem:two scales} because the fact that $\omega$ is concave with $\omega(0)=0$ implies that the function $(t>0)\mapsto \omega(t)/t$ is nonincreasing, whence  for distinct $x,y\in \MM$ we have:
 $$
\frac{\omega\left(d_\MM(x,y)+d_\MM(\{x,y\},\sub)\right)}{d_\MM(x,y)+d_\MM(\{x,y\},\sub)}\le \frac{\omega\left(d_\MM(x,y)\right)}{d_\MM(x,y)}.
$$ So, \eqref{eq:omega extended}  implies that $F$ has Lipschitz constant $O(\ee(\MM;\cZ))$ when viewed as a function from $\omega\circ \MM$ to $\cZ$.

Prior to proving Lemma~\ref{lem:two scales}, we will next  show how it implies Proposition~\ref{prop:def to ext}:

\begin{proof}[Deduction of Proposition~\ref{prop:def to ext}  from Lemma~\ref{lem:two scales}] Fix $0<\theta\le 1$ and $K\ge 1$. Suppose that  $\Id_\MM:\MM\to \MM$ admits a $K$-de-$\theta$-H\"oldering factorization through  $\cZ$. This means that there are $g:\MM\to \cZ$ and $h:\cZ\to \MM$ with: 
\begin{equation}\label{eq:composition identity}
h\circ g=\Id_\MM,
\end{equation} 
such that the following two conditions hold for some $\alpha>0$: 
\begin{equation}\label{eq:factor}
\left\{\begin{array}{ll}
\forall x,y\in \MM,\qquad d_\cZ\big(g(x),g(y)\big)\le \alpha d_\MM(x,y)^{\theta},\\ \forall z,w\in \cZ,\qquad d_\MM\big(h(z),h(w)\big)^\theta \le \frac{K}{\alpha}  \max \big\{d_\cZ(z,w),  d_\cZ\big(\{z,w\},g(\MM)\big)^{1-\theta}d_\cZ(z,w)^\theta\big\}.\end{array}\right. 
\end{equation}
 Letting $(\SS,d_\SS)$ be a metric space satisfying $\ee(\SS;\cZ)<\infty$, our goal is to deduce that:   
 \begin{equation}\label{eq:restate extension goal}
 \ee(\SS;\MM)^\theta\lesssim K\ee(\SS;\cZ).
 \end{equation} 

Suppose that  $\emptyset\neq \sub\subset \SS$ is closed and that $f:\sub\to \MM$ is Lipschitz. By the first part of~\eqref{eq:factor} we have:
$$
\forall s,t\in \sub,\qquad d_\cZ\big(g\circ f(s),g\circ f(t)\big)\le \alpha \|f\|_{\Lip(\sub;\MM)}^\theta d_\SS(s,t)^\theta. 
$$
Consequently,  Lemma~\ref{lem:two scales} applied to $g\circ f:\sub\to \cZ$, provides $\Gamma:\SS\to \cZ$  that satisfies:
\begin{equation}\label{eq:Gamma composition}
\forall s\in \sub,\qquad \Gamma(s)=g\circ f(s),
\end{equation}
as well as: 
\begin{equation}\label{eq:Gamma}
\forall s,t\in \SS,\qquad  d_\cZ\big(\Gamma(s),\Gamma(t)\big)\lesssim   \frac{\ee(\SS;\cZ)\alpha \|f\|_{\Lip(\sub;\MM)}^\theta d_\SS(s,t) }{\left(d_\SS(s,t)^{\phantom{2}}\!\!\!\!+d_\SS(\{s,t\},\sub)\right)^{1-\theta}}\le \ee(\SS;\cZ)\alpha \|f\|_{\Lip(\sub;\MM)}^\theta d_\SS(s,t)^\theta.
\end{equation}

 Observe that:
$$
\forall s\in \SS,\qquad d_\cZ\big(\Gamma(s),g(\MM)\big)\le  d_\cZ\big(\Gamma(s),g\circ f(\sub)\big)\stackrel{\eqref{eq:Gamma composition}}{=} d_\cZ\big(\Gamma(s),\Gamma(\sub)\big)\stackrel{\eqref{eq:Gamma}}{\lesssim} \ee(\SS;\cZ) \alpha \|f\|_{\Lip(\sub;\MM)}^\theta  d_\SS(s,\sub)^\theta, 
$$
where the first step holds because $f(\sub)\subset \MM$. Consequently:
\begin{equation}\label{eq:st}
\forall s,t\in \SS,\qquad d_\cZ\big(\{\Gamma(s),\Gamma(t)\},g(\MM)\big)\lesssim  \ee(\SS;\cZ) \alpha \|f\|_{\Lip(\sub;\MM)}^\theta   d_\SS(\{s,t\},\sub)^\theta,
\end{equation}
and therefore the following estimate holds for every $s,t\in \SS$:
$$
\max\Big\{d_\cZ\big(\Gamma(s),\Gamma(t)\big),d_\cZ\big(\{\Gamma(s),\Gamma(t)\},g(\MM)\big)^{1-\theta}d_\cZ\big(\Gamma(s),\Gamma(t)\big)^\theta\Big\}
\stackrel{\eqref{eq:Gamma}\wedge\eqref{eq:st}}{ \lesssim} \ee(\SS;\cZ)\alpha \|f\|_{\Lip(\sub;\MM)}^\theta  d_\SS(s,t)^\theta.
$$
In combination with~\eqref{eq:factor}, this implies that: 
$$
\forall s,t\in \SS,\qquad d_\MM\big(h\circ \Gamma(s),h\circ \Gamma(t)\big)^\theta\lesssim K\ee(\SS;\cZ) L^\theta d_\SS(s,t)^\theta.
$$
In other words, the Lipschitz constant of $h\circ \Gamma:\SS\to \MM$ satisfies: 
$$
\|h\circ  \Gamma\|_{\Lip(\SS;\MM)}^\theta\lesssim K\ee(\SS;\cZ)\|f\|_{\Lip(\sub;\MM)}^\theta.
$$
So, $\ee(\SS,\sub;\MM)^\theta\lesssim K\ee(\SS;\cZ)$, as   $h\circ \Gamma$ extends $f$ thanks to~\eqref{eq:composition identity} and~\eqref{eq:Gamma composition},  and $f$ is an arbitrary Lipschitz function from $\sub$ to $\MM$. Because $\sub$ is an arbitrary closed subset of $\SS$, the desired estimate~\eqref{eq:restate extension goal} follows.
\end{proof}

\begin{proof}[Proof of Lemma~\ref{lem:two scales}] For each $k\in \Z$ fix a $2^k$-net $\NN_k\subset \sub$ of $\sub$, i.e., $d_\MM(a,b)\ge 2^k$ for all distinct $a,b\in \NN_k$, and for every $a\in \sub$ there is $c\in \NN_k$ such that $d_\MM(a,c)\le 2^k$. It follows from this that:
\begin{equation}\label{eq:distance from outside to net in C}
\forall x\in \MM,\qquad d_\MM(x,\NN_k)\le d_\MM(x,\sub)+2^k.
\end{equation}

Distinct $a,b\in \NN_k$ satisfy $d_\cZ(f(a),f(b))\le \omega(d_\MM(a,b))\le \omega(2^k)d_\MM(a,b)/2^k$, using $d_\MM(a,b)\ge 2^k$ and the fact that $(t\in (0,\infty))\mapsto \omega(t)/t$ is nonincreasing, by the concavity of $\omega$ and $\omega(0)=0$.     So, the restriction of $f$ to $\NN_k$ is Lipschitz with constant $\omega(2^k)/2^k$, whence there is $\Phi_k:\MM\to \cZ$ satisfying: 
\begin{equation}\label{eq:extended from net}
\forall (x,y,a)\in \MM\times \MM\times \NN_k,\qquad d_\cZ\big(\Phi_k(x),\Phi_k(y)\big)\lesssim \ee(\MM;\cZ)\frac{\omega(2^k)}{2^k}d_\MM(x,y)\qquad\mathrm{and}\qquad \Phi_k(a)=f(a). 
\end{equation}

We record in passing (for ease of later use) the following simple estimate:
\begin{equation}\label{eq:increment}
\forall x\in \MM,\qquad d_\cZ\big(\Phi_k(x),\Phi_{k+1}(x)\big)\lesssim \ee(\MM;\cZ)\frac{\omega(2^k)}{2^k} \big(d_\MM(x,\sub)+2^k\big).
\end{equation}
To explain why~\eqref{eq:increment} holds, take $a\in \sub$ such that $d_\MM(x,a)\lesssim d_\MM(x,\sub)$, as well as $u\in \NN_k$ and $v\in \NN_{k+1}$ such that $d_\MM(u,a)\le 2^k$ and $d_\MM(v,a)\le 2^{k+1}$.  Then:
\begin{multline*}
d_\cZ\big(\Phi_k(x),\Phi_{k+1}(x)\big)\stackrel{\eqref{eq:extended from net}}{\le} d_\cZ\big(\Phi_k(x),\Phi_k(u)\big)+d_\cZ\big(f(u),f(a)\big)+d_\cZ\big(f(a),f(v)\big)+d_\cZ\big(\Phi_{k+1}(v),\Phi_{k+1}(x)\big)\\
\stackrel{\eqref{eq:omega modulus}\wedge \eqref{eq:extended from net}}{\lesssim} \ee(\MM;\cZ)\frac{\omega(2^k)}{2^k} d_\MM(x,u)+\omega\big(d_\MM(u,a)\big)+\omega\big(d_\MM(v,a)\big)+\ee(\MM;\cZ)\frac{\omega(2^{k+1})}{2^k} d_\MM(v,x),
\end{multline*}
which implies~\eqref{eq:increment}, because $d_\MM(u,a),d_\MM(v,a)\lesssim 2^k$,  by the triangle inequality and the above choice of $a$ we have $d_\MM(x,u),d_\MM(x,v)\lesssim d_\MM(x,a)+2^k\lesssim d_\MM(x,\sub)+2^k$, and because $\omega(O(1)s)\lesssim \omega(s)$ for every $s>0$. 

Fix a conical geodesic bicombing $\gamma$ on $\cZ$, i.e., it is a function as in~\eqref{eq:bicombing1} that satisfies~\eqref{eq:bicombing2} and~\eqref{eq:bicombing3}. Define:
\begin{equation}\label{our extended F}
\forall x\in \MM\setminus \sub,\qquad F(x)\eqdef \gamma\Big(\Phi_{k(x)}(x),\Phi_{k(x) +1}(x),\frac{d_\MM(x,\sub)}{2^{k(x)}}-1\Big),\quad\mathrm{where} \quad  k(x)\eqdef \left\lfloor \log_2 d_\MM(x,\sub)\right\rfloor\in \Z.
\end{equation}
 Note that~\eqref{our extended F} makes sense as $\sub$ is closed and $x\notin \sub$, so $d_\MM(x,\sub)>0$.  Also set $F(a)=f(a)$ for  $a\in \sub$. Our goal is  to  check~\eqref{eq:omega extended}. If $\{x,y\}\subset \sub$, then~\eqref{eq:omega extended} follows from~\eqref{eq:omega modulus}, even without  the factor $\ee(\MM;\cZ)\ge 1$.  Hence, we may assume without loss of generality that $x,y\in \MM$ are distinct, $d_\MM(x,\sub)>0$ and $d_\MM(x,\sub)\ge d_\MM(y,\sub)$.

If $y\in \sub$, then $F(y)=f(y)$ and our goal~\eqref{eq:omega extended} becomes $d_\cZ(F(x),f(y))\lesssim \ee(\MM;\cZ)\omega(d_\MM(x,y))$. To see why this indeed holds, observe first that:
\begin{align}\label{eq:y in C}
\begin{split}
d_\cZ\big(F(x),f(y)\big)&\stackrel{\eqref{our extended F}}{=}d_\cZ\bigg(\gamma\Big(\Phi_{k(x)}(x),\Phi_{k(x) +1}(x),\frac{d_\MM(x,\sub)}{2^{k(x)}}-1\Big),\gamma\Big(f(y),f(y),\frac{d_\MM(x,\sub)}{2^{k(x)}}-1\Big)\bigg)\\&\stackrel{\eqref{eq:bicombing3}}{\le} \Big(2-\frac{d_\MM(x,\sub)}{2^{k(x)}}\Big)
d_\cZ\big(\Phi_{k(x)}(x),f(y)\big)+\Big(\frac{d_\MM(x,\sub)}{2^{k(x)}}-1\Big)d_\cZ\big(\Phi_{k(x)+1}(x),f(y)\big)\\
&\le 
d_\cZ\big(\Phi_{k(x)}(x),f(y)\big)+d_\cZ\big(\Phi_{k(x)+1}(x),f(y)\big).
\end{split}
\end{align}
The first summand in the right hand side of~\eqref{eq:y in C} can be bounded as follows:
\begin{align}
\begin{split}\label{eq:kx near y}
d_\cZ\big(\Phi_{k(x)}(x),f(y)\big)&\le \inf_{a\in \NN_{k(x)}} \Big(d_\cZ\big(\Phi_{k(x)}(x),\Phi_{k(x)}(a)\big)+d_\cZ\big(f(a),f(y)\big)\Big)\\& \lesssim
\inf_{a\in \NN_{k(x)}} \Big(\ee(\MM;\cZ)\frac{\omega(2^{k(x)})}{2^{k(x)}}d_\MM(x,a)+\omega\big(d_\MM(a,x)+d_\MM(x,y)\big)\Big)\\& \lesssim \ee(\MM;\cZ)\omega\big(d_\MM(x,y)\big)+\omega\big(3d_\MM(x,y)\big)\lesssim \ee(\MM;\cZ)\omega\big(d_\MM(x,y)\big),
\end{split}
\end{align}
where the first step of~\eqref{eq:kx near y} holds by the triangle inequality and the second part of~\eqref{eq:extended from net},  the second step of~\eqref{eq:kx near y} uses~\eqref{eq:omega modulus} and~\eqref{eq:extended from net}, as well as $d_\MM(a,y)\le d_\MM(a,x)+d_\MM(x,y)$ and  that $\omega$ is nondecreasing, the third step of~\eqref{eq:kx near y}  holds as $2^{k(x)}\asymp d_\MM(x,\sub)$ and  $d_\MM(x,\NN_{k(x)})\le d_\MM(x,\sub) +2^{k(x)}\le 2d_\MM(x,\sub)\le 2d_\MM(x,y)$ using~\eqref{eq:distance from outside to net in C}, the definition of $k(x)$ in~\eqref{our extended F} and $y\in \sub$, and the final step of~\eqref{eq:kx near y} holds since $\omega(O(1)s)\lesssim \omega(s)$ for every $s>0$.     The second summand in the right hand side of~\eqref{eq:y in C} is controlled analogously: 
\begin{align}
\begin{split}\label{eq:kx+1 near y}
d_\cZ\big(\Phi_{k(x)+1}(x),f(y)\big)&\le \inf_{a\in \NN_{k(x)+1}} \Big(d_\cZ\big(\Phi_{k(x)+1}(x),\Phi_{k(x)+1}(a)\big)+d_\cZ\big(f(a),f(y)\big)\Big)\\& \lesssim
\inf_{a\in \NN_{k(x)+1}} \Big(\ee(\MM;\cZ)\frac{\omega(2^{k(x)+1})}{2^{k(x)+1}}d_\MM(x,a)+\omega\big(d_\MM(a,x)+d_\MM(x,y)\big)\Big)\\& \lesssim \ee(\MM;\cZ)\omega\big(2d_\MM(x,y)\big)+\omega\big(4d_\MM(x,y)\big)\lesssim \ee(\MM;\cZ)\omega\big(d_\MM(x,y)\big).
\end{split}
\end{align}
The desired estimate~\eqref{eq:omega extended} when $y\in \sub$ now follows by substituting~\eqref{eq:kx near y} and~\eqref{eq:kx+1 near y} into~\eqref{eq:y in C}.

If $d_\MM(x,y)\ge d_\MM(y,\sub)=d_\MM(\{x,y\},\sub)$, then the desired estimate~\eqref{eq:omega extended} would follow if we demonstrate that $d_\cZ(F(x),F(y))\lesssim \ee(\MM;\cZ)\omega(d_\MM(x,y))$. The latter inequality  indeed holds because we may fix $a,b\in \sub$ such that $d_\MM(x,a)\lesssim d_\MM(x,\sub)\le d_\MM(x,y)+d_\MM(y,\sub)\le 2d_\MM(x,y)$ and $d_\MM(y,b) \lesssim d_\MM(y,\sub)\le d_\MM(x,y)$, whence also $d_\MM(a,b)\le d_\MM(a,x)+d_\MM(x,y)+d_\MM(y,b)\lesssim d_\MM(x,y)$, and conclude that:
\begin{align}\label{eq:different ks}
\begin{split}
d_{\cZ}\big(F(x),F(y)\big)\le d_{\cZ}\big(F(x),&F(a)\big)+d_{\cZ}\big(f(a),f(b)\big)+d_{\cZ}\big(F(b),F(y)\big)\\&\lesssim \ee(\MM;\cZ)\omega\big(O(1)d_\MM(x,y)\big)\lesssim \ee(\MM;\cZ)\omega\big(d_\MM(x,y)\big),
\end{split}
\end{align}
where the penultimate step of~\eqref{eq:different ks} is an instantiation of the (already proven) case $\{x,y\}\cap \sub\neq\emptyset$ of~\eqref{eq:omega extended}, together with the aforementioned bounds $d_\MM(x,a),d_\MM(y,b),d_\MM(a,b)\lesssim d_\MM(x,y)$.

The remaining (and most significant) case of~\eqref{eq:omega extended} is  $d_\MM(x,y)< d_\MM(y,\sub)$, so the integers $k(y)\le k(x)$ from~\eqref{our extended F} are well defined, and  determined by $2^{k(y)}\le d_\MM(y,\sub)<2^{k(y)+1}$ and  $2^{k(x)}\le d_\MM(x,\sub)<2^{k(x)+1}$.  Observe that $2^{k(x)}\le d_\MM(x,\sub)\le d_\MM(x,y)+d_\MM(y,\sub)<2d_\MM(y,\sub)<2^{k(y)+2}$, whence  $k(x)\in \{k(y),k(y)+1\}$. Writing from now $k=k(y)$, we thus have $k(x)\in \{k,k+1\}$  and $2^{k}\asymp d_\MM(x,\sub)\asymp d_\MM(y,\sub)=d_\MM(\{x,y\},\sub)$. Furthermore, because $\omega(O(1)2^k)\lesssim \omega(2^k)$, the desired inequality~\eqref{eq:omega extended} becomes:
\begin{equation}\label{eq:reduced goal interesting case}
 d_\cZ\big(F(x),F(y)\big)\lesssim \ee(\MM;\cZ)\frac{\omega(2^k)}{2^k}d_\MM(x,y). 
\end{equation}
We will prove~~\eqref{eq:reduced goal interesting case}  by treating the last two remaining cases $k(x)=k$ and $k(x)=k+1$ separately.

Suppose firstly that $k(x)=k=k(y)$. Then: 
\begin{align*}
d_\cZ\big(F(x),F(y)\big) &\!\!\stackrel{\eqref{our extended F}}{\le} d_\cZ\bigg(\gamma\Big(\Phi_{k}(x),\Phi_{k +1}(x),\frac{d_\MM(x,\sub)}{2^{k}}-1\Big),\gamma\Big(\Phi_{k}(x),\Phi_{k +1}(x),\frac{d_\MM(y,\sub)}{2^{k}}-1\Big)\bigg)\\&\qquad\qquad +d_\cZ\bigg(\gamma\Big(\Phi_{k}(x),\Phi_{k +1}(x),\frac{d_\MM(y,\sub)}{2^{k}}-1\Big),\gamma\Big(\Phi_{k}(y),\Phi_{k +1}(y),\frac{d_\MM(y,\sub)}{2^{k}}-1\Big)\bigg)\\
&\!\!\!\!\!\!\!\stackrel{\eqref{eq:bicombing2}\wedge\eqref{eq:bicombing3}}{\le} \frac{d_\MM(x,\sub)-d_\MM(y,\sub)}{2^k}d_\cZ\big(\Phi_k(x),\Phi_{k+1}(x)\big)+d_\cZ\big(\Phi_k(x),\Phi_k(y)\big)+
d_\cZ\big(\Phi_{k+1}(x),\Phi_{k+1}(y)\big)\\&\!\!\!\!\!\!\!\!\!\stackrel{\eqref{eq:extended from net}\wedge \eqref{eq:increment}}{\lesssim}
\frac{d_\MM(x,y)}{2^k}\ee(\MM;\cZ)\frac{\omega(2^k)}{2^k} 2^k+
\ee(\MM;\cZ)\frac{\omega(2^k)}{2^k}d_\MM(x,y)+\ee(\MM;\cZ)\frac{\omega(2^{k+1})}{2^{k+1}}d_\MM(x,y)\\&\lesssim \ee(\MM;\cZ)\frac{\omega(2^k)}{2^k}d_\MM(x,y),
\end{align*}
where the penultimate  step uses the triangle inequality through $d_\MM(x,\sub)-d_\MM(y,\sub)\le d_\MM(x,y)$. This proves~\eqref{eq:reduced goal interesting case} when $k(x)=k$. 

Justifying the final case  $k(x)=k+1=k(y)+1$ is analogous and only a bit more involved as it requires  one additional triangle inequality step as follows:
\begin{align*}
d_\cZ\big(F(x),F(y)\big) &\!\!\stackrel{\eqref{our extended F}}{\le} d_\cZ\bigg(\gamma\Big(\Phi_{k+1}(x),\Phi_{k +2}(x),\frac{d_\MM(x,\sub)}{2^{k+1}}-1\Big),\Phi_{k+1}(x)\bigg)\\&\qquad\qquad+d_\cZ\bigg(\Phi_{k+1}(x),\gamma\Big(\Phi_{k}(x),\Phi_{k +1}(x),\frac{d_\MM(y,\sub)}{2^{k}}-1\Big)\bigg)\\&\qquad\qquad +d_\cZ\bigg(\gamma\Big(\Phi_{k}(x),\Phi_{k +1}(x),\frac{d_\MM(y,\sub)}{2^{k}}-1\Big),\gamma\Big(\Phi_{k}(y),\Phi_{k +1}(y),\frac{d_\MM(y,\sub)}{2^{k}}-1\Big)\bigg)\\
&=d_\cZ\bigg(\gamma\Big(\Phi_{k+1}(x),\Phi_{k+2}(x),\frac{d_\MM(x,\sub)}{2^{k+1}}-1\Big),\gamma\Big(\Phi_{k+1}(x),\Phi_{k+2}(x),0\Big)\bigg)
\\&\qquad\qquad+d_\cZ\bigg(\gamma\Big(\Phi_{k}(x),\Phi_{k +1}(x),1\Big),\gamma\Big(\Phi_{k}(x),\Phi_{k +1}(x),\frac{d_\MM(y,\sub)}{2^{k}}-1\Big)\bigg)\\&\qquad\qquad +d_\cZ\bigg(\gamma\Big(\Phi_{k}(x),\Phi_{k +1}(x),\frac{d_\MM(y,\sub)}{2^{k}}-1\Big),\gamma\Big(\Phi_{k}(y),\Phi_{k +1}(y),\frac{d_\MM(y,\sub)}{2^{k}}-1\Big)\bigg)\\
&\!\!\!\!\!\!\!\stackrel{\eqref{eq:bicombing2}\wedge\eqref{eq:bicombing3}}{\le} \frac{d_\MM(x,\sub)-2^{k+1}}{2^{k+1}}d_\cZ\big(\Phi_{k+1}(x),\Phi_{k+2}(x)\big)+\frac{2^{k+1}-d_\MM(y,\sub)}{2^{k}}
d_\cZ\big(\Phi_{k}(x),\Phi_{k+1}(x)\big)
\\ &\qquad\qquad+d_\cZ\big(\Phi_k(x),\Phi_k(y)\big)+
d_\cZ\big(\Phi_{k+1}(x),\Phi_{k+1}(y)\big)\\&\!\!\!\!\!\!\!\!\!\stackrel{\eqref{eq:extended from net}\wedge \eqref{eq:increment}}{\lesssim}
\frac{d_\MM(x,\sub)-2^{k+1}}{2^k}\ee(\MM;\cZ)\omega(2^{k+1})+\frac{2^{k+1}-d_\MM(y,\sub)}{2^k}\ee(\MM;\cZ)\omega(2^{k})\\&\qquad\qquad +
\ee(\MM;\cZ)\frac{\omega(2^k)}{2^k}d_\MM(x,y)+\ee(\MM;\cZ)\frac{\omega(2^{k+1})}{2^{k+1}}d_\MM(x,y)\\&\lesssim 
\frac{d_\MM(x,\sub)-d_\MM(y,\sub)}{2^k}\ee(\MM;\cZ)\omega(2^{k})+\ee(\MM;\cZ)\frac{\omega(2^k)}{2^k}d_\MM(x,y)\lesssim \ee(\MM;\cZ)\frac{\omega(2^k)}{2^k}d_\MM(x,y).
\end{align*}

The last part of Lemma~\ref{lem:two scales} that $F$ can be taken to depend linearly on $f$ when $\cZ$ is a Banach space follows  from its definition~\eqref{our extended F} if we take $\gamma$ to be the canonical convex combination bicombing,  for each $k\in \Z$ we let $\Phi_k$ be the image of $f|_{\NN_k}$ under a suitable fixed linear extension operator (corresponding to extension from $\NN_k$), and with $\ee_l(\MM;\cZ)$ replacing $\ee(\MM;\cZ)$ in the right hand side of~\eqref{eq:extended from net}.    
\end{proof}

\subsection*{Added in proof} Forthcoming work with Manor Mendel~\cite{MN26} answers the question posed in Remark~\ref{rem:Mcotype}  by demonstrating  that $L_1$ has metric Markov cotype $2$. Also, \cite{MN26} confirms our prediction in the discussion following Corollary~8 by showing that indeed $\ee(L_q;L_1)\lesssim \sqrt{q}$ for every $q\ge 2$.

\bibliographystyle{alphaabbrvprelim}
\bibliography{deH}

\end{document}